\documentclass[11pt,reqno]{amsart}

\usepackage{amssymb}
\usepackage{amsmath}
\usepackage{amsfonts} 
\usepackage{amsbsy}
\usepackage{amscd}
\usepackage{amsthm}
\usepackage{geometry}    
\usepackage{courier}
\usepackage{xcolor}

\usepackage{graphicx}
\usepackage{amssymb}
\usepackage{epstopdf}
\usepackage{mathtools}
\usepackage{mathrsfs}
\usepackage[capitalize]{cleveref}
\usepackage{mathrsfs} 
\usepackage[framemethod=tikz]{mdframed} 
\usepackage{circuitikz}
\usetikzlibrary{decorations.pathmorphing, patterns,shapes}

\usetikzlibrary{arrows}
\usetikzlibrary{patterns}
\usepackage{xcolor}
\usepackage{caption}
\usepackage{subcaption}

\newcommand{\Z}{\mathbb{Z}}

\newcommand{\Hidden}[1]{}

\newtheorem{theorem}{Theorem}

\newtheorem{proposition}[theorem]{Proposition}

\theoremstyle{definition}
\newtheorem{remark}[theorem]{Remark}

\newtheorem{example}[theorem]{Example}
\theoremstyle{remark}

\title[Proof of the Linear 3-tree Conjecture]{Proof of a Conjecture on the Growth of the Maximal Resistance Distance in a Linear 3--Tree}
\author{Emily J. Evans}
\address{Brigham Young University}
\email{EJEvans@math.byu.edu}
\author{Russell Jay Hendel}
\address{Towson University}
\email{RHendel@Towson.Edu}

\begin{document}

\begin{abstract} In~\cite{bef} Barret, Evans, and Francis conjectured that if $G$ is the straight linear 3-tree with $n$ vertices and $H$ is the straight linear 3-tree with $n+1$ vertices then 
	\[\lim_{n\rightarrow \infty} r_{H} (1, n+1) - r_G(1,n) = \frac{1}{14},\]
where $r_G(u,v)$ and $r_H(u,v)$ are the resistance distance between vertices $u$ and $v$ in graphs $G$ and $H$ respectively.  In this paper, we prove the conjecture by looking at the determinants of deleted Laplacian matrices.  The proof uses a Laplace expansion method on a family of determinants to determine the underlying recursion this family satisfies and then uses routine linear algebra methods to obtain an exact Binet formula for the $n$-th term. 
\end{abstract}

\maketitle

\section{Introduction}
Resistance distance, also referred to as effective resistance, is a well-known metric on graphs that measures both the number of paths between two vertices in a graph and the cost of each path.  A wide variety of applications of resistance distance exist including applications to mathematical connectivity in social, biological, ecological, and transportation networks ~\cite{Applications}, mathematical chemistry~\cite{rdmatrix,carmona2014effective,Cinkir,KleinRandic, klein2002resistance,klein2004random, kem1,peng2017kirchhoff, yang2014comparison, yang2008kirchhoff}, graph theory~\cite{bapatdvi,BapatWheels,MarkK, DEVRIENDT202224,littleswim, Ghosh, klein1997graph, ZHOU20172864}, numerical linear algebra~\cite{SpielSparse}, and engineering~\cite{Barooah06grapheffective}.  Of special interest to graph theorists and mathematical chemists is the calculation of resistance distance in families of graphs.  We say a graph is a member of a family if a particular structure is maintained as the number of nodes grows.

Given a graph $G$ the resistance distance between two nodes is determined by considering the graph as an electric circuit where each edge is represented by a resistor whose resistance is the inverse of the edge weight. Given any two nodes $i$ and $j$ assume that one unit of current flows into node $i$ and one unit of current flows out of node $j$.  The potential difference $v_i - v_j$ between nodes $i$ and $j$ needed to maintain this current is the {\it resistance distance} between $i$ and $j$. 

In this paper, we address an open conjecture of Barret, Evans, and Francis~\cite{bef} regarding the asymptotic behavior of the resistance distance in a so-called straight linear 3-tree.  We recall that a straight linear $3$-tree is a graph $G_n$ on $n$ vertices whose adjacency matrix is symmetric, banded, with the first through third super and sub diagonals equal to one, and all other entries equal to zero.  In other words $G_n$ is the graph whose (0,1)-adjacency matrix is defined by 
\[
a_{ij}=\begin{cases}1 &\text{if $0 < |i-j| \leq 3$}\\0 & \text{otherwise.}\end{cases}\]
In~\cite{bef} it was conjectured that if $G$ was the straight linear 3-tree with $n$ vertices and $H$ was the straight linear 3-tree with $n+1$ vertices then 
\begin{equation}\label{equ:conjecturetoprove}
[\lim_{n\rightarrow \infty} r_{H} (1, n+1) - r_G(1,n) = \frac{1}{14},
\end{equation}
where $r_G(u,v)$ and $r_H(u,v)$ are the effective resistance between vertices $u$ and $v$ in graphs $G$ and $H$ respectively.

In~\cite{bef} circuit transformations were used to obtain a result similar to the conjectured result for linear 2-trees.  Many other possible techniques exist to determine resistance distance in graphs including the use of matrices (the combinatorial Laplacian) and graph theoretic approaches. For a summary, including worked examples see~\cite{littleswim}.

Among methods that use the Laplacian matrix, one technique uses determinants associated with the underlying matrix with specific rows and columns deleted~\cite{bapatdvi}.  In this paper, we approach the computation of a determinant by calculating the recursion satisfied by the underlying family. These recursions allow us to compute Binet forms and consequently compute closed-formula for resistances. The use of recursive relationships satisfied  by families of determinants is not a new idea; many well-known formulas exist
(e.g., tridiagonal matrices~\cite{ELMIKKAWY2004669}, pentadiagonal matrices~\cite{pentadiag,recursion},
block tridiagonal matrices~\cite{molinari2008determinants}, and Toeplitz matrices~\cite{li2011calculating}).  By performing a Laplace expansion to calculate the determinants, \cite{recursion} shows that the determinants of the general pentadiagonal family of matrices governed by five parameters, satisfy a sixth-order recursion whose roots can be explicitly calculated. This allows the authors to calculate a closed form for the values of determinants of these matrices; the authors show that the resulting formula for resistance distance is much quicker than other methods for calculating resistance distance.

\section{Notation and important background results}
We recall that the Laplacian matrix associated with a graph is defined to be $L=D-A$ where $A$ is the adjacency matrix (in our case a symmetric heptadiagonal matrix) and $D$ is a diagonal matrix whose entries correspond to the degree of each vertex.  Since we will be handling families of square matrices, we will use a superscript to denote the size of the matrix, e.g., $L^n$ will correspond to a Laplacian matrix on a graph with $n$ vertices.  Similarly, $L^{n}_{i,j}$ refers to the entry in row $i$ column $j$ of $L^{n}.$  

If $A$ and $B$ are sets of indices (or singleton indices) and $M$ is an arbitrary matrix, then  $M^{n}(A|B)$ is the matrix obtained from $M^{n}$ by deleting the rows whose indices are in $A$ and deleting the columns whose indices are in $B.$  If the sets are singletons we will drop the braces for notational simplicity.  Notice that in interpreting say 
$L^{n}(1|m)$ the operation of taking the $n$-th member of the family is performed prior to deleting the row and column; thus the resulting matrix has size $n-1 \times n-1.$  $Det(M^{n})$ denotes the determinant of $M^{n}.$ 

We make use of a formula from Bapat~\cite{bapatdvi} to compute the resistance distance between nodes $u$ and $v$ of $L^{n}.$ Since we are interested in the maximal resistance distance in a given straight linear 3-tree of size $n$, we let  $u$ and $v$ correspond to nodes 1 and $n$ where nodes 1 and $n$ are the two vertices of degree 3. More precisely the equation we will use to determine resistance distance is:
\begin{equation}\label{equ:Bapat}
        \text{Resistance distance between nodes 1 and $n$} = 
        \frac{Det(L^n(\{1,n\}|\{1,n\}))}{Det(L^n(1|1))}.
\end{equation}
Equation \eqref{equ:Bapat} is also valid if the denominator is $Det(L^n(n|n)).$

\section{Overview of the Laplace Expansion Procedure}
As mentioned in the introduction, we calculate the resistance of the linear 3-tree by mimicking the methods presented in \cite{recursion}. In that paper, Laplace Expansion was applied to the family of pentadiagonal matrices. This resulted in a set of equations in determinant families. The resulting set of equations was then solved to yield a recursive relationship satisfied by a family of determinants. Using standard techniques, the characteristic polynomial corresponding to the recursion may be used to compute a Binet form which allows explicit computation of the resistance.

This procedure is done in \cite{recursion} without any explicit formalism for the procedure.  The method when applied to the linear 3-tree is more complex, and therefore, some formalization is useful.  

The Laplace expansion of a matrix $M^{n}$ along the first row is given by  
\begin{equation}\label{equ:basicsteprow}
    Det(M^{n}) = \displaystyle \sum_{M^{n}_{1,j} \neq 0} (-1)^{j+1} M^{n}_{1,j}  Det(M^{n}(1|j)), 
\end{equation}
with a similar formula for expansion along the first column. To restate \eqref{equ:basicsteprow} in terms of family of matrices we use the backward shift operator $y$ which operates on a sequence $\{s_n\}$ (of complex numbers) by $y s_n = s_{n-1}$ (note that we indicate the operation of $y$ by simple juxtaposition without parenthesis).  Thus \eqref{equ:basicsteprow} may be rewritten
\begin{equation}\label{equ:basicstepfamily}
    Det(M) = \displaystyle \sum_{M_{1,j} \neq 0} (-1)^{j+1} M_{1,j} y Det(M(1|j)), 
\end{equation}
which is a purely formal notation interpreted to mean that for each $n-th$ matrix of the underlying families the resulting equation in determinants is true.

\subsection{The Laplace Expansion Procedure}

The pseudo-code for the Laplace Expansion procedure is presented below; however, a simple example  is presented first to clarify the concepts involved. For this example, we use the path graph, which may be represented as a straight line with $n$ vertices with the corner vertices having degree 1 and other vertices having degree 2. For $n=6,$  the Laplacian is given by \[L^{6}=\left(\begin{smallmatrix}1 & -1 & 0&0&0&0\\-1&2&-1&0&0&0\\0&-1&2&-1&0&0\\0&0&-1&2&0&0\\0&0&0&-1&2&-1\\0&0&0&0&-1&1 \end{smallmatrix}\right).\] 

The Bapat formula \eqref{equ:Bapat} requires  computing recursions separately for the numerators and denominators. We suffice for this illustrative example of the Laplace expansion procedure with treating the denominator, whose underlying matrix of size $4 \times 4$ is given by \[D^{4}=L^6(\{1,6\}|\{1,6\})=\left(\begin{smallmatrix}2 & -1 & 0&0\\-1&2&-1&0\\0&-1&2&-1\\0&0&-1&2
\end{smallmatrix}\right).\]

\begin{example}\label{exa:path}  The procedure, consisting of a series of Laplace expansions.  Each Laplace expansion affects three queues: 
\begin{itemize}
    \item \textbf{QTodo} which stores the matrix families needing an expansion operation,
    \item \textbf{QEquations} which stores the resulting determinant identities,  and 
    \item \textbf{QDone} which stores all matrix families already used in the procedure.
\end{itemize}   The matrices in \textbf{QDone} are sequentially indexed. We initialize the process by setting \textbf{QTodo}=$\{D(0)\}$ (with $D(0)=D$), and setting the queues \textbf{QDone} and \textbf{QEquations} to be the empty set.  The process then proceeds as follows

\begin{enumerate}

\item[\textbf{Step 1:}]\textbf{QTodo}=$\{D(0)\}$ so we add $D(0)$ to $\textbf{QDone}$, and remove it from $\textbf{QTodo}.$  Then using \eqref{equ:basicstepfamily} we expand $D(0)$ as  
$D(0) = 2 yD(0)(1|1) + y D(0)(1|2).$ We notice that 
$D(0)(1|1) = D(0)$ while $D(0)({1|2}) =\left(\begin{smallmatrix}-1&-1&0&0\\0&2&-1&0\\0&-1&2&-1\\0&0&-1&2
\end{smallmatrix}\right)$ is new. We set
$D(0)({1|2}) = D(1),$ place $D(1)$ in \textbf{QTodo} and place $D(0) = 2 yD(0) + y D(1)$ in \textbf{QEquations.}

\item[\textbf{Step 2:}] \textbf{QTodo} = $\{D(1)\}.$ We add $D(1)$ to $\textbf{QDone},$  and remove it from $\textbf{QTodo}.$  We then Laplace Expand  
$D(1) = -yD(1)({1|1}).$ We notice that 
$D(1)(1|1) = D(0).$  We set
$D(1)({1|2}) = D(0),$  and place $D(1) = - yD(0)$  in \textbf{QEquations.}
\end{enumerate}

Since \textbf{QTodo} is empty the procedure terminates. Although the procedure is not guaranteed to terminate it does terminate for the linear 3-tree. 
\end{example}
 The set of equations in \textbf{QEquations} will be solved in the next section. First we present the pseudo code for the Laplace expansion process and discuss its memory and run time requirements.

\begin{verbatim}
PROCEDURE LaplaceExpand (Takes a matrix family $M(0)=M$ 
and produces, as needed, a set of new matrix families and 
a set of equations in matrix families. 
The process may or may not terminate)
INITIALIZE:
    QTodo = {M(0)}
    QDone = { }
    QEquations = { }
    Counter=0

WHILE Not IsEmpty(QTodo)
    FOR EACH q in QTodo
        DeleteFrom(QTodo, q)
        AddTo(QDone,q)
        FOR EACH m in LaplaceExpandMatrices(q)
            IF Exists i, m =M(i) in QDone THEN
                 Rename(m, M(i))
            ELSE Rename(m, M(Counter++))
                 AddTo(QTodo,m)
                 AddTo(QEquations, LaplaceExpandEquations(q))
            END IF
        END FOR EACH
    END FOR EACH
END WHILE
\end{verbatim}

The sub-procedure 
\textbf{LaplaceExapndEquations} refers to \eqref{equ:basicstepfamily}.

Both authors independently reduced by hand the matrix families associated with the Laplacian of the linear 3-tree (for the numerator and denominator).  Additionally, to ensure correctness, a Matlab and Mathematica program were written to check work. For those who wish to experiment with it, the Mathematica program is presented as an Appendix in \cite{Automated}. All four attempts produced the same underlying recursion. 
 
 However, there are areas to improve efficiency. The authors' attempts resulted in 20 new matrix families, while the Mathematica attempts resulted in 80 new matrix families (The difference between the humans and software lies in the choice of the row or the column on which to Laplace expand; the software only expanded along the first row or column depending on which had fewer 
 non-zero entries).
 The size of the matrices was adjusted to 10, since it preserved all features of the Laplacian of the linear 3-tree. Therefore the memory requirements for even 80 matrices is insignificant. Because the process is formal with no numerical computations of determinants, the run time is also insignificant.

\section{Solving the Set of Determinant Equations}
Returning to Example \ref{exa:path}, when the procedure terminated there were two equations in the \textbf{QEquations} queue: $D(0)=2y D(0)+y D(1)$ and
$D(1)= -y D(0).$ Substituting the left-hand side of the second equation into the first equation yields an equation in one matrix family, $(y^2-2y+1)D(0)=0.$ Recalling the interpretation of $y$ as a backward shift operator, this equation implies that for all $n$ where defined,
$D^{n}(0) - 2 D^{n-1}(0) + D^{n-2}(0).$ Thus  $y^2-2y+1$ is a characteristic polynomial for the family $D(0).$ We may then derive the Binet form using standard methods.

This method of solution is applicable to the general case.

\begin{theorem} A system of $n \ge 1$ identities in families of determinants,
\begin{equation}\label{eq:intheorem}
    D(i) = \sum_{j=1}^n a_{i,j} D(j), \qquad  1 \le i \le n
\end{equation}
with the $a_{i,j}$ integer polynomials  in a variable $y$
may be solved, that is, may be reduced to a single identity in one matrix family. 
\end{theorem}
\begin{proof} If for $1\le i \le n$
$a_{i,i} \neq 0,$ then for each $i$ we may rewrite \eqref{eq:intheorem} as
$$
    (1-a_{i,i}) D(i) = \sum_{j=1, j \neq i}^n a_{i,j} D(j).
$$
Thus we may replace the system of $n$ determinant identities with the following equivalent system.
\begin{equation}\label{equ:system}
   b_{i,i} D(i) = \sum_{j=1, j \neq i}^n b_{i,j} D(j), \qquad  1 \le i \le n,
\end{equation}
with $b_{i,j}$ integer polynomials in $y.$ 

We use the phrase the $k$-th equation to refer to the equation with $i=k$ in \eqref{equ:system}.  
If $n>1,$ we first multiply the $i$-th equation, $1 \le i \le n-1,$ by $b_{n,n}.$ We may then substitute the right hand side  of the $n$-th equation for  the summands of the form $b_{i,n} b_{n,n}  D(n)$ of the resulting $i$-th equations, $1 \le i \le n-1.$ We thus  obtain a system of equations similar to \eqref{equ:system} except that it consists of $n-1$ equations in $n-1$ determinant identities with coefficients of integer polynomials in $y.$

We may continue this reduction process till we are left with one equation in one determinant family. Collecting terms we then have an equation of the form $p D(0) =0$
 with $p$ an integer polynomial in $y.$ It follows that $p(y)$ is a characteristic polynomial for the sequence $D(0).$
\end{proof}

\begin{remark} Because $\Z$  is a principle ideal domain, the minimal polynomial for a sequence generates the ideal of characteristic polynomials. It follows that if we have any characteristic polynomial we can test each factor to ascertain if it is the minimal polynomial. The following proposition provides such a test.
\end{remark}

\begin{proposition}\label{pro:minimalpolynomial} Let $A(y),$$ B(y),$  and $C(y)$ denote polynomials in the backward shift operator $y.$ Further let $x=\{x_n\}_{n \ge -\infty},$ be an arbitrary doubly infinite sequence of complex numbers.  If the following three conditions hold: (i) $A(y)$ annihilates $x$  (ii)  $A(y)= B(y) C(y),$ and (iii) For some integer $n_1,$ $C(x_n) = 0, \text{ for } n=n_1,n_1+1,\dotsc, n_1+deg(B)-1$, then $C$ annihilates $x.$ \end{proposition}
\begin{proof}   We know that $A=BC$ annihilates $x.$ Hence,  $B$ annihilates the sequence $Cx$ (i.e. the operator $C$ applied to the sequence $x$).  Since $B$ is of degree $deg(B),$ corresponding to a recursion of order $deg(B)$ satisfied by the sequence $\{Cx\},$  values of this sequence are determined by any $deg(B)$ consecutive values. It follows that if $Cx=0$ on $deg(B)$ consecutive terms then it must identically equal 0. \end{proof}

\begin{remark} The requirement of a doubly infinite sequence is not restrictive. If the sequence $\{x_n\}_{n \ge c}$ for some integer constant $c$ satisfies a recursion then we may extend this sequence to  a doubly infinite sequence by applying the underlying recursion backward.\end{remark}

\section{Closed Formula for Resistance in the Linear 3-tree} This section applies the methods of the previous sections to the linear 3-tree providing closed forms that will allow us to prove \eqref{equ:conjecturetoprove},
 and additionally provide an exact formula for the underlying sequence which allows computation of error terms.

The $n \times n$ Laplacian matrix, $L^{n},$ of the straight linear 3--tree is given by 
\begin{equation}\label{equ:PLinear3}
L_G=\begin{bmatrix}	
3 & -1 & -1 & -1 & \dotsc  &0 &0 &  0 & 0\\
-1 & 4 & -1 & -1 & \dotsc  &0 &0 &  0 & 0 \\ 
-1 & -1 & 5 & -1 &\dotsc   &0 & 0 & 0 & 0\\
-1 & -1 & -1 & 6 &\dotsc   &0 & 0 & 0 & 0  \\
\vdots & \ddots & \ddots & \ddots &\ddots &\ddots
\ddots & \ddots &  \vdots \\
0 & 0 & 0 & 0 & \dotsc &6  &-1  &-1 & -1  \\
0 & 0 & 0 & 0 & \dotsc &-1  &5  & -1 & -1  \\
0 & 0 & 0 & 0 & \dotsc &-1  &-1 &  4 & -1 \\
0 & 0 & 0 & 0 & \dotsc &-1  &-1  & -1 & 3\\
\end{bmatrix}.
\end{equation}
The center of the matrix continues with sixes on the diagonal and negative ones on the first three super and sub diagonals.    

To apply \eqref{equ:Bapat} we first define the numerator and denominator matrix sequences, by
$$
    D^{n}= L^{n+2}(\{1,n\}|\{1,n\}) \qquad
    N^{n} =L^{n+1}(1|1).
$$

Using this notation, \eqref{equ:Bapat} becomes
\begin{equation}
\label{equ:Bapat1}
 \text{Resistance distance between nodes $1$ and $n$}
 = 
 \frac{Det(N^{n-1})}{Det({D^{n-2})}}.
\end{equation}

Applying the methods of the previous sections, the sequence $\{Det(D^{n})\}_{n \ge 5}$ satisfies the recursion corresponding to the characteristic polynomial (or equivalently the annihilator) $D(X)$ given by
\begin{equation}\label{equ:denominator}
    D(X)=(X-1) \left(X^4-4 X^3-X^2-4 X+1\right),
\end{equation}
with roots
\begin{multline*}
\{r_1,r_2,r_3,r_4,r_5\}= 
\left\{1,\frac{1}{2} \left(-\sqrt{7}-i \sqrt{4 \sqrt{7}-7}+2\right),\frac{1}{2} \left(-\sqrt{7}+i \sqrt{4 \sqrt{7}-7}+2\right),\right.\\ \left.\frac{1}{2} \left(\sqrt{7}-\sqrt{4 \sqrt{7}+7}+2\right),\frac{1}{2} \left(\sqrt{7}+\sqrt{4 \sqrt{7}+7}+2\right)\right\}
\end{multline*}
while the sequence $\{Det(N^{n})\}_{n \ge 4}$ satisfies the recursion corresponding to the characteristic polynomial, $N(X)$ given by
$$
 N(X)=(X-1)^2 \left(X^4-4 X^3-X^2-4 X+1\right)^2 \left(X^4+3 X^3+6 X^2+3 X+1\right)
$$
with roots  
\begin{multline*}
\{r_1,r_2,\dotsc,r_9\}= \left\{r_1,\dotsc, r_5,  
\frac{i \sqrt{7}}{4}-\frac{1}{2} \sqrt{-\frac{7}{2}-\frac{1}{2} 3 i \sqrt{7}}-\frac{3}{4}, 
\frac{i \sqrt{7}}{4}+\frac{1}{2} \sqrt{-\frac{7}{2}-\frac{1}{2} 3 i \sqrt{7}}-\frac{3}{4},\right.\\\left.-\frac{i \sqrt{7}}{4}-\frac{1}{2} \sqrt{-\frac{7}{2}+\frac{3 i \sqrt{7}}{2}}-\frac{3}{4}, -\frac{i \sqrt{7}}{4}+\frac{1}{2} \sqrt{-\frac{7}{2}+\frac{3 i \sqrt{7}}{2}}-\frac{3}{4} \right\},
\end{multline*}
with the multiplicity of $r_1,\dotsc,r_5$ equaling 2.

To clarify the limits, the recursion $Det(D^{n+5}) = 5 Det(D^{n+4}) - 3 Det(D^{n+3})+ 3 Det(D^{n+2})-5 Det(D^{n+1}) + Det (D^{n})$  is true for $n=5$ but not for $n<5.$ (The corresponding cut-off point  for the numerator sequence is $n=4$.) This is due to degeneracies in the Laplacian for small $n.$ Attempts to calculate the Binet Form can create numerical overflow errors if the roots are large. To avoid this we define sister sequences of integers
$\tilde{D}^{n}$ and $\tilde{N}^{n}$ agreeing with $Det(D^{n})$ and $Det(N^{n})$ for $n \ge 5$ and $n  \ge 4$ respectively but satisfying the underlying recursions for all integers. The construction of these sister sequences is done by simply applying the respective underlying recursions backward. 

It is then straightforward to solve for the coefficients in the Binet form for the sister sequences. Since the sister sequences agree with the original sequences for $n \ge 5$ and $n \ge 4$ we will have closed forms for them as well. 

For the sequence $\{\tilde{D}^{n}\}$ we solve the system of
5 equations in 5 unknowns, $ \sum_{i=1}^5 d_i r_i^j, j \in \{0,1,2,3,4\}$ and obtain 
\begin{multline*} \{d_1,\dotsc, d_5 \} = 
 \left\{-\frac{8}{7}, \frac{1}{42} \left(-6 \sqrt{7}+5 i \sqrt{4 \sqrt{7}+7}-2 i \sqrt{7 \left(4 \sqrt{7}+7\right)}+12\right),\right.\\
\frac{1}{42} \left(-6 \sqrt{7}-5 i \sqrt{4 \sqrt{7}+7}+2 i \sqrt{7 \left(4 \sqrt{7}+7\right)}+12\right),\\ \frac{1}{42} \left(5 \sqrt{4 \sqrt{7}-7}+2 \left(3 \sqrt{7}+\sqrt{7 \left(4 \sqrt{7}-7\right)}+6\right)\right),\\ \left.\frac{1}{42} \left(6 \left(\sqrt{7}+2\right)-5 \sqrt{4 \sqrt{7}-7}-2 \sqrt{7 \left(4 \sqrt{7}-7\right)}\right)\right\}.
\end{multline*}
As can be seen these roots lie in a quartic extension of the rationals which consist of a tower of quadratic extensions.  

Similarly, for the numerator, we solve the system of 14 equations in 14 unknowns
$$ \sum_{i=1}^9 n_i(j) r_i^j, j \in \{1, \dotsc, 14\},$$ with the $n_i$ constant integer functions  for $i=6,7,8,9$ and linear polynomials, say, $n_i(X)=n_{i,0} + n_{i,1} x$ for $i=1,2,3,4,5.$ This gives the following set of coefficients ( residing in a tower of quadratic extensions of the rationals):
\begin{multline*}
\{n_{0,0},n_{0,1},
\dotsc,
n_{5,0},n_{5,1},
n_6,\dotsc,n_9\}=
\left\{-\frac{12}{49},-\frac{4}{49},\right.\\
\frac{-210 \sqrt{7}-28 i \sqrt{4 \sqrt{7}-7}+155 i \sqrt{7 \left(4 \sqrt{7}-7\right)}+693}{8232},  
\frac{-9 \sqrt{7}-i \sqrt{4 \sqrt{7}-7}+2 i \sqrt{7 \left(4 \sqrt{7}-7\right)}+24}{1176},\\
\frac{-210 \sqrt{7}+28 i \sqrt{4 \sqrt{7}-7}-155 i \sqrt{7 \left(4 \sqrt{7}-7\right)}+693}{8232},
\frac{-9 \sqrt{7}+i \sqrt{4 \sqrt{7}-7}-2 i \sqrt{7 \left(4 \sqrt{7}-7\right)}+24}{1176},\\
\frac{210 \sqrt{7}-28 \sqrt{4 \sqrt{7}+7}-155 \sqrt{7 \left(4 \sqrt{7}+7\right)}+693}{8232}, 
\frac{9 \sqrt{7}-\sqrt{4 \sqrt{7}+7}-2 \sqrt{7 \left(4 \sqrt{7}+7\right)}+24}{1176},\\
\frac{210 \sqrt{7}+28 \sqrt{4 \sqrt{7}+7}+155 \sqrt{7 \left(4 \sqrt{7}+7\right)}+693}{8232},\frac{9 \sqrt{7}+\sqrt{4 \sqrt{7}+7}+2 \sqrt{7 \left(4 \sqrt{7}+7\right)}+24}{1176},\\
\frac{\sqrt{-1+\frac{3 i}{\sqrt{7}}} \left(4 \sqrt{7}-7\right) \left(4 \sqrt{7}+7\right) \left(\sqrt{2} \left(41-27 i \sqrt{7}\right)-9 \sqrt{-7-3 i \sqrt{7}}-21 i \sqrt{7 \left(-7-3 i \sqrt{7}\right)}\right)}{98784},\\
\frac{i \sqrt{-1+\frac{3 i}{\sqrt{7}}} \left(4 \sqrt{7}-7\right) \left(4 \sqrt{7}+7\right) \left(\sqrt{2} \left(27 \sqrt{7}+41 i\right)+9 i \sqrt{-7-3 i \sqrt{7}}-21 \sqrt{7 \left(-7-3 i \sqrt{7}\right)}\right)}{98784}, \\
\frac{\sqrt{-1-\frac{3 i}{\sqrt{7}}} \left(4 \sqrt{7}-7\right) \left(4 \sqrt{7}+7\right) \left(\sqrt{2} \left(41+27 i \sqrt{7}\right)-9 \sqrt{-7+3 i \sqrt{7}}+21 i \sqrt{7 \left(-7+3 i \sqrt{7}\right)}\right)}{98784}, \\
\left.\frac{i \sqrt{-1-\frac{3 i}{\sqrt{7}}} \left(4 \sqrt{7}-7\right) \left(4 \sqrt{7}+7\right) \left(\sqrt{2} \left(-27 \sqrt{7}+41 i\right)+9 i \sqrt{-7+3 i \sqrt{7}}+21 \sqrt{7 \left(-7+3 i \sqrt{7}\right)}\right)}{98784}\right\}.
\end{multline*}

By \eqref{equ:Bapat} or \eqref{equ:Bapat1}, a closed formula for $R(n)$, the resistance between nodes 1 and $n$ of the linear 3-tree of size $n,$ is given by
\begin{equation}\label{equ:Rexact}
R(n) =
\frac{\sum_{i=1}^9 n_i(n-2)
r_i^{n-2}}
{\sum_{i=1}^5 d_i r_i^{n-1}}.
\end{equation}
 
The following facts about the roots will prove useful in obtaining asymptotic forms and error terms. We observe that $r_1,r_2,r_3$ lie on the unit circle in the complex plane and hence may be ignored when obtaining asymptotic and asymptotic error terms. Similarly $r_4,r_7,r_9$ lie inside the unit disk and may also be ignored in asymptotic calculations. The dominant root is $|r_5| \approx 4.42$ and suffices for the calculation of the asymptotic limit.  However $|r_6| \approx 2.1$ and
$|r_8| \approx 2.1$ have absolute values greater than one implying they contribute to the error term. 

To complete the proof of \eqref{equ:conjecturetoprove}, note that by
  \eqref{equ:Rexact} and the comments on the roots just made, we may define a function
$$
R^{Asymptotic}(n) = 
\frac{n_{5}(n-2) r_5^{n-2}}{d_5 r_5^{n-1}},
$$
with 
$$
R^{Asymptotic}(n) \approx
R(n).
$$

But then
$$
R(n+1)-R(n) \approx
R^{Asymptotic}(n+1) -
R^{Asymptotic}(n) =
\frac{n_{5,1}}{d_5 r_5}
= \frac{1}{14}.
$$

An error term for $R(n)$ may be obtained using the roots lying outside the unit circle as discussed above showing that convergence is geometric. 

\section{Conclusion}
This paper has formalized a procedure introduced by 
\cite{recursion} and applied it to the linear 3-tree. This allows proof of a conjecture by Barret, Evans, and Francis and also provides an exact formula for the resistance distance between node 1 and node $n$ in a linear 3-tree.

The formal procedure introduced seems to have independent interest in its own right and may be applicable to a wider variety of graph families whose adjacency matrices are banded (or nearly banded). Whether the procedure converges, as well as how one might improve the efficiency of it remain open questions.


\vspace{1cm}
\textbf{Disclosure statement.} The authors declare that they have no conficting financial interests with the results of this paper.

\end{document}